\documentclass[12pt]{amsart}
\usepackage[utf8]{inputenc}

\font\myfont=cmr12 at 20pt

\title[Hardy's uncertainty principle for the Wigner and Schr\"{o}dinger evolutions]{\myfont{\textbf{Notes on Hardy's Uncertainty Principle for the Wigner distribution and Schr\"{o}dinger evolutions}}}
\author{Helge Knutsen}
\email{helge.knutsen@ntnu.no}
\date{November 2022}
\address{Department of Mathematical Sciences, Norwegian University of Science and \mbox{Technology,} \newline 7034 Trondheim, Norway}

\usepackage{amsthm}
\theoremstyle{plain} 

\newtheorem{theorem}{Theorem}[section]
\newtheorem{corollary}{Corollary}[section]
\newtheorem{lemma}[theorem]{Lemma}

\theoremstyle{definition} 
\newtheorem*{remark}{Remark}
\newtheorem{example}{Example}[section]

\usepackage{amsmath, amssymb} 
\numberwithin{equation}{section} 

\usepackage{mathtools}

\usepackage{enumitem} 

\usepackage{calrsfs} 
\DeclareMathAlphabet{\pazocal}{OMS}{zplm}{m}{n}

\usepackage{bm} 

\usepackage[toc,page]{appendix} 

\usepackage{graphicx}
\graphicspath{ {figures/} }

\usepackage{pst-node}
\usepackage{tikz-cd}

\usepackage{geometry}
\usepackage{afterpage}

\usepackage{verbatim} 

\usepackage{subfigure}

\usepackage[hyphens]{url}

\usepackage{csquotes}

\usepackage[colorlinks,urlcolor=black]{hyperref} 

\usepackage{tikz-cd} 

\begin{document}

\maketitle


\begin{abstract}
    We consider Schr\"{o}dinger equations with real quadratic Hamiltonians, for which the Wigner distribution of the solution at a given time equals, up to a linear coordinate transformation, the Wigner distribution of the initial condition. Based on Hardy's uncertainty principle for the joint time-frequency representation, we prove a uniqueness result for such Schr\"{o}dinger equations, where the solution cannot have strong decay at two distinct times. This approach reproduces known, sharp results for the free Schr\"{o}dinger equation and the harmonic oscillator, and we also present an explicit scheme for quadratic systems based on positive definite matrices.
\end{abstract}

\section{Introduction}
Hardy's uncertainty principle is originally formulated as a sharp decay estimate of a function $f$ and its Fourier transform $\hat{f}$. Normalizing the Fourier transform by
\begin{align}
    \hat{f}(\xi) :=(2\pi)^{-n}\int_{\mathbb{R}^n}f(x)e^{-i\langle \xi,x\rangle}\mathrm{d}x,
    \label{normalization_Fourier_transform}
\end{align}
it states that the decay of $(f,\hat{f})$ cannot exceed 
\begin{align*}
    |f(x)| = \pazocal{O}(e^{-\alpha|x|^2}) \ \ \text{and} \ \ |\hat{f}(\xi)| = \pazocal{O}(e^{-\beta|\xi|^2}) \ \ \text{for} \ \ 4 \alpha\beta=1.
\end{align*}
Later, at the beginning of this century in \cite{Escauriaza_Kenig_Ponce_Vega_2006} and \cite{Chanillo_2007}, Hardy's uncertainty principle has taken a different interpretation, where the statement has been shown to be equivalent to a sharp uniqueness result on the free Schr\"{o}dinger equation
\begin{align}
    \frac{\partial u}{\partial t}(x,t) = i\Delta u(x,t).
    \label{simple_free_Schrodinger_equation}
\end{align}
More precisely, at two distinct times $t=0$ and $t=T$, the solution $u$ of \eqref{simple_free_Schrodinger_equation} cannot exceed the decay conditions 
\begin{align*}
    |u(x,0)| = \pazocal{O}(e^{-\alpha|x|^2}) \ \ \text{and} \ \ |u(x,T)| = \pazocal{O}(e^{-\beta|x|^2}) \ \ \text{for} \ \ (4T)^2 \alpha\beta =1.
\end{align*}
For an in-depth discussion on this dynamical interpretation and different proofs of Hardy's uncertainty principle, we refer to the recent survey-paper \cite{Bertolin_Malinnikova_2021}. Starting with \cite{Escauriaza_Kenig_Ponce_Vega_2006}, and in the sequel of papers \cite{Escauriaza_Kenig_Ponce_Vega_2008}, \cite{Escauriaza_Kenig_Ponce_Vega_2008_JEMS}, \cite{Escauriaza_Kenig_Ponce_Vega_2010}, Escauriaza, Kenig, Ponce and Vega have studied Schr\"{o}dinger evolutions for more general Hamiltonians, which include a bounded potential. Their scheme is based on establishing logarithmic-convexity properties of the solution of the Schr\"{o}dinger equation, from which they successfully derive uniqueness results similar to the free case. In the same spirit, in \cite{Cassano_Fanelli_2015} and more recently in \cite{Cassano_Fanelli_2017}, Cassano and Fanelli consider Schr\"{o}dinger evolutions of the harmonic oscillator and of systems with a magnetic potential, in addition to some bounded perturbations of these. In particular, the magnetic potential is given by $\Delta_{A}:=\big(\nabla-i A(x)\big)^2$ for some coordinate transformation $A:\mathbb{R}^n\to\mathbb{R}^n$, so that the associated Schr\"{o}dinger equation (with a bounded perturbation $V:\mathbb{R}^n\times \mathbb{R}\to \mathbb{C}$) reads
\begin{align*}
    \frac{\partial u}{\partial t}(x,t) = i\big(\Delta_{A}+V(x,t)\big)u(x,t).
\end{align*}
With the exception of some specific examples, and under some restraints on the coordinate transform $A$, the uniqueness results for the magnetic potential are all derived for dimension $n\geq 3$. Note that, if we assume $A$ to be a linear transformation (and disregard any perturbations), both the harmonic oscillator and the magnetic potential represent systems with a \textit{quadratic} Hamiltonian, which is the focus of the present paper. 

For real quadratic Hamiltonians (without any perturbations) we present an \textit{alternative} approach based on Hardy's uncertainty principle for the Wigner distribution. Here we utilize that for such Schr\"{o}dinger equations, the Wigner distribution of the solution equals, up to a linear coordinate transform, the Wigner distribution of the initial condition. This description of the solution in terms of the Wigner distribution first came to our attention through the works of Cordero and Rodino, and later Cordero, Giacchi and Rodino, in their two-part series \cite{Cordero_Rodino_2022_prtI}, \cite{Cordero_Giacchi_Rodino_2022_prtII} on the Wigner analysis of operators. From here it turns out to be remarkably simple to formulate a general uniqueness results for quadratic systems. Nonetheless, considering specified quadratic Hamiltonians, it can still be demanding to produce explicit Hardy type estimates. 

The remainder of the text is organized as follows: In Section \ref{section_preliminaries}, we cover necessary background theory. Namely, we introduce the Weyl quantization procedure (Section \ref{section_Weyl_quantization}), and introduce relevant theory from symplectic mechanics (Section \ref{section_symplectic_mechanics}), which seems to be a natural perspective when considering real quadratic quantum systems. We also briefly discuss Hardy's uncertainty principle in relation to the Wigner distribution (Section \ref{section_Hardys_UP_Wigner_distribution}). 
Section \ref{section_Hardys_UP_for_quadratic_Hamiltonians} contains the general uniqueness result, where the subsequent and final Section \ref{section_examples_of_Schrodinger_evolutions} is devoted to specific examples of Schr\"{o}dinger evolutions. In Section \ref{section_examples_of_Schrodinger_evolutions}, we reproduce known sharp Hardy type estimates for both the free case and the harmonic oscillator. In addition, based on Williamson's diagonalization theorem, we present an explicit scheme for systems based on real, symmetric positive definite matrices, with one final example for dimension $n=2$.  

\section{Preliminaries}
\label{section_preliminaries}
\subsection{Weyl quantization and the Schr\"{o}dinger equation} 
\label{section_Weyl_quantization}
For the position and momentum observables $x, p \in\mathbb{R}^{n}$, we associate the following pseudodifferential operator
\begin{align*}
    x_j \to \widehat{x}_j=x_j \ \ \text{and} \ \ p_j\to \widehat{p}_j = -i\hbar \frac{\partial}{\partial x_j}.
\end{align*}
For their composition $x_j p_k$, we consider standard Weyl quantization 
\begin{align*}
    x_j p_k \to \frac{1}{2}\big(\widehat{x}_j\widehat{p}_k+\widehat{p}_k \widehat{x}_j\big). 
\end{align*}
Although we only require the quantization procedures outlined above, we mention that in the general case the Weyl pseudodifferential operator $\widehat{H}$ with symbol function $H = H(x,p)$ is given by
\begin{align}
    \widehat{H}f(x)= (2\pi \hbar)^{-n}\int_{\mathbb{R}^{2n}} H\left(\frac{x+y}{2},\xi\right)e^{\frac{i}{\hbar}\langle x-y, \xi\rangle}f(y)\mathrm{d}y\mathrm{d}\xi.
    \label{Weyl_quantization_procedure}
\end{align}
In the context of the quantum mechanics, we consider symbol functions that represent the Hamiltonian of a physical system. The Schr\"{o}dinger equation with Hamiltonian $H$ is then given by
\begin{align}
    i \hbar \frac{\partial u}{\partial t}(x,t) = \widehat{H}u(x,t),
    \label{Schrodinger_equation}
\end{align}
where the solution $u$ is a function (or distribution) of the variables $x$ and $t$, subject to some initial condition, e.g., $u(x,0) = u_{0}(x)$. We shall focus on the case when the Hamiltonian is \textit{quadratic}, namely, $H(x,p) = \sum_{j,k}a_{jk}x_j p_k$ for real $a_{jk}$'s. Alternatively, this can be expressed as 
\begin{align}
    H(z) = \frac{1}{2}\langle M z, z\rangle \ \ \text{for} \ \ z=(x,p),
    \label{quadratic_hamiltonian_symmetric_matrix}
\end{align}
and $M$ is some real-valued $2n\times 2n$ symmetric matrix. On this form, \textit{symplectic mechanics} naturally enters the picture, and we reference \cite{deGosson_Symplectic_Methods} for a comprehensive introduction.

\subsection{Symplectic mechanics} 
\label{section_symplectic_mechanics}
We briefly cover some of the basic terminology and results from symplectic mechanics necessary to describe solutions of the Schr\"{o}dinger equation \eqref{Schrodinger_equation} with quadratic Hamiltonians \eqref{quadratic_hamiltonian_symmetric_matrix}. This is based on de Gosson's book \cite{deGosson_Symplectic_Methods}, specifically chapter 1-3, 7 and 15: 

Let $J$ denote the standard symplectic matrix 
\begin{align*}
    J:= \begin{pmatrix}0 & I \\ -I & 0\end{pmatrix},
\end{align*}
where $0$ and $I$ are the $n\times n$ zero and identity matrices, respectively. A useful observation is that the inverse of $J$ coincides with the transpose so that $J^{-1} = J^{T} = -J$. The \textit{symplectic (Lie) group} $\mathrm{Sp}(2n,\mathbb{R})$ is a closed subgroup of the general linear group $GL(2n,\mathbb{R})$, that consists of all matrices $S$ such that
\begin{align*}
    S^{T}J S = J. 
\end{align*}
Since the inverse $S^{-1}$ is also symplectic, this latter condition turns out to be equivalent to $S J S^{T}=J$.
Writing the matrix $S$ on block form
\begin{align*}
    S = \begin{pmatrix} A & B\\ C & D\end{pmatrix},
\end{align*}
it is straightforward to verify that symplectic matrices are characterized by the conditions
\begin{equation}
    \begin{aligned}
    A^{T}C, B^{T}D \ \ &\textit{are symmetric and} \ \ A^{T}D-C^{T}B = I\\
    &\ \ \ \ \ \ \ \ \ \ \ \ \iff\\
    A B^{T}, C D^{T} \ \ &\textit{are symmetric and} \ \ AD^{T}-BC^{T} = I. 
    \end{aligned}
    \label{symplectic_block_form_conditions}
\end{equation}
From this characterization, we also deduce that the inverse of $S$ is given by
\begin{align}
    S^{-1}= \begin{pmatrix}D^{T} & -B^{T}\\ -C^{T} & A^{T}\end{pmatrix}.
    \label{symplectic_matrix_inverse}
\end{align}
The standard symplectic form is denoted by $\sigma(z,z') := \langle Jz, z'\rangle$.
A basis $\{e_j, f_j\}_{j=1}^{n}$ of $\mathbb{R}^{n}\times \mathbb{R}^{n}$ is called a \textit{symplectic basis} if 
\begin{align*}
    \sigma(e_j, e_k) = 0 \ \ \land \ \ 
    \sigma(f_j, f_k) = 0 \ \ \land \ \ 
    \sigma(e_j, f_k) = -\delta_{jk} \ \ \forall \ \ j,k=1,2,\dots,n.
\end{align*}
One simple example is the \textit{canonical} symplectic basis $\{(c_j,0),(0,c_j)\}_{j=1}^{n}$, where $\{c_j\}_j$ is the canonical basis for $\mathbb{R}^{n}$. 

An important family of symplectic matrices are the so-called "free" symplectic matrices. These are matrices $S\in \mathrm{Sp}(2n,\mathbb{R})$ which on block form satisfy
\begin{align}
    S = \begin{pmatrix}A & B\\ C & D\end{pmatrix}, \ \ \det B \neq 0,
    \label{free_symplectic_matrix}
\end{align}
and every symplectic matrix can in fact be written as a product of two such matrices. To each free matrix $S$ on the form \eqref{free_symplectic_matrix}, we associate a quadratic form $W(x,x')$, called the generating function of $S$, given by
\begin{align*}
    W(x,x'):=\frac{1}{2}\langle DB^{-1} x, x\rangle - \langle B^{-1}x, x'\rangle + \frac{1}{2}\langle B^{-1}A x', x'\rangle.
\end{align*}
Conversely, starting out with a quadratic form 
\begin{align}
    W(x,x') = \frac{1}{2}\langle P x, x\rangle - \langle L^{-1}x, x'\rangle +\frac{1}{2}\langle Qx', x'\rangle,
    \label{generating_function_starting_point}
\end{align}
with real matrices such that $P = P^{T}$ and $Q=Q^{T}$, we can generate a corresponding free symplectic matrix
\begin{align}
    S_{W} = \begin{pmatrix} LQ & L \\ PLQ -(L^{-1})^{T} & PL\end{pmatrix}.
    \label{free_sympletic_matrix_from_generating_function}
\end{align}
Based on the generating function $W$, we associate unitary operators on $L^2(\mathbb{R}^{n})$ to the free matrix $S_{W}$. Referencing the generating function in \eqref{generating_function_starting_point} and by scaling according to the constant $\hbar$, these operators read 
\begin{align}
    \widehat{S}_{W,m}u(x) := (2\pi i \hbar)^{-\frac{n}{2}} \sqrt{|\det L|}^{-1} i^{m} \int_{\mathbb{R}^{n}} e^{\frac{i}{\hbar}W(x,x')}u(x')\mathrm{d}x',
    \label{quadratic_Fourier_transform}
\end{align}
where $m$ is an integer so that 
\begin{align*}
    m\pi \equiv \arg(\det L)\mod 2\pi. 
\end{align*}
We refer to the operators $\widehat{S}_{W,m}$ as the \textit{quadratic Fourier transforms} of $S_{W}$, and by fixing a branch of $\sqrt{z}$ so that $\sqrt{i}$ is well-defined, we obtain \textit{two} operators associated to $S_{W}$ that only differ by a sign. In particular, we choose the branch such that $\arg(\sqrt{i})\equiv \frac{\pi}{4} \mod 2\pi$. Since the inverse exists
\begin{align*}
    \widehat{S}_{W,m}^{-1} = \widehat{S}_{W^{*},m^{*}} \ \ \text{for} \ \ W^{*}(x,x') = -W(x',x) \ \ \text{and} \ \ m^{*}=m-n,
\end{align*}
we can form a group from compositions of the quadratic Fourier transforms of the free symplectic matrices. This group is known as the \textit{metaplectic group} $\mathrm{Mp}(2n,\mathbb{R})$, and it forms a double cover of the symplectic group $\mathrm{Sp}(2n, \mathbb{R})$ by extending the map $\widehat{S}_{W,m}\to S_{W}$ to a surjective group homomorphism
\begin{align*}
    \pi^{\mathrm{Mp}}:\mathrm{Mp}(2n,\mathbb{R})\to \mathrm{Sp}(2n,\mathbb{R}), \ \ \text{with} \ \ \ker(\pi^{\mathrm{Mp}})= \{\pm I\}. 
\end{align*}

To see how the symplectic and metaplectic group relates to the Schr\"{o}dinger equation with quadratic Hamiltonian $H$, we first recall the general Hamiltonian equations determining the time evolution of a point $z =(x,p)$, namely
\begin{align*}
    \frac{\partial x_j}{\partial t} = \frac{\partial H}{\partial p_j} \ \ \text{and} \ \ \frac{\partial p_j}{\partial t} = -\frac{\partial H}{\partial x_j} \ \ \text{for} \ \ j=1,2,\dots,n, 
\end{align*}
or more compactly
\begin{align*}
    \frac{\partial z}{\partial t} = J\nabla H \ \ \text{with} \ \ \nabla = (\partial_{x_1},\dots, \partial_{x_n}, \partial_{p_1}, \dots, \partial_{p_n}).  
\end{align*}
When the Hamiltonian is on the form $H(z) = \frac{1}{2}\langle Mz, z\rangle$, where $M$ is a real-symmetric matrix, the Hamiltonian equations reduce to $\frac{\partial z}{\partial t} = J M z$. Starting at time $t=0$ from the point $z(0)=z_{0}$, it is clear that $z(t) = S^{H}_{t}(z_0)$ with $S^{H}_{t} := \exp(tJM)$ solves the initial value problem. 

We now consider the \textit{Lie algebra of the symplectic group}, also known as the \textit{symplectic algebra}, denoted by $\mathfrak{sp}(2n,\mathbb{R})$, which consists of all matrices $X$ such that
\begin{align*}
    XJ+JX^{T}=0.
\end{align*}
By inspection, we find that the matrix $JM$ (and also $tJM$) belongs to $\mathfrak{sp}(2n,\mathbb{R})$. Conversely, for any $X\in \mathfrak{sp}(2n,\mathbb{R})$, we have that $JX$ is symmetric. Hence, any quadratic Hamiltonian $H$ can be expressed as
\begin{align}
    H(z) = -\frac{1}{2}\langle JX z, z\rangle \ \ \text{for some} \ \ X\in \mathfrak{sp}(2n,\mathbb{R}). 
    \label{quadratic_hamiltonian_symplectic_algebra}
\end{align}
Since the exponential maps the symplectic \textit{algebra} into the symplectic \textit{group}, we must have that the operator $S_{t}^{H}=\exp(tX)$ is symplectic for any fixed $t$, and in turn $\left(S_{t}^{H}\right)_{t\geq 0}$ is a one-parameter subgroup of $\mathrm{Sp}(2n,\mathbb{R})$. 
 
We also consider the \textit{Lie algebra of the metaplectic group}, or simply the \textit{metaplectic algebra}, denoted by $\mathfrak{mp}(2n,\mathbb{R})$. An essential observation is that the metaplectic algebra $\mathfrak{mp}(2n,\mathbb{R})$ is \textit{isomorphic} to $\mathfrak{sp}(2n,\mathbb{R})$, which follows from the fact that $\mathrm{Mp}(2n,\mathbb{R})$ is a covering group of $\mathrm{Sp}(2n,\mathbb{R})$. In particular, there is an explicit isomorphism $F: \mathfrak{sp}(2n,\mathbb{R})\to \mathfrak{mp}(2n,\mathbb{R})$ so that the following diagram commutes:
\begin{equation}
    \begin{tikzcd}
        \mathfrak{mp}(2n,\mathbb{R}) \arrow{r}{F^{-1}} \arrow[swap]{d}{\exp} & \mathfrak{sp}(2n,\mathbb{R}) \arrow{d}{\exp} \\
        \mathrm{Mp}(2n,\mathbb{R}) \arrow{r}{\pi^{\mathrm{Mp}}}& \mathrm{Sp}(2n,\mathbb{R})
\end{tikzcd}
\label{commuting_diagram_metaplecitc_symplectic_algebra_group}
\end{equation}
What follows are the main results of the section. For details and proofs, we refer to chapter 15 in \cite{deGosson_Symplectic_Methods} (especially Theorem 352 and Corollary 355). 
\begin{theorem}
Let $H$ be the quadratic Hamiltonian of the form $H(z) = -\frac{1}{2}\langle JX z, z\rangle$ for some $X\in \mathfrak{sp}(2n,\mathbb{R})$. The linear mapping $F$, which to $X$ associates the operator $F(X):=-\frac{i}{\hbar}\widehat{H}$ where $\widehat{H}\xleftrightarrow{\textnormal{Weyl}}H$, then satisfies 
\begin{align*}
    [F(X), F(X')] = F([X,X']) \ \ \forall \ \ X,X'\in \mathfrak{sp}(2n, \mathbb{R}),   
\end{align*}
and $F$ forms an isomorphism $\mathfrak{sp}(2n,\mathbb{R})\to \mathfrak{mp}(2n, \mathbb{R})$
so that diagram \eqref{commuting_diagram_metaplecitc_symplectic_algebra_group} commutes. 
\end{theorem}
 
Based on the isomorphism $F$ and since diagram \eqref{commuting_diagram_metaplecitc_symplectic_algebra_group} commutes, we can describe the solution of Schr\"{o}dinger equation \eqref{Schrodinger_equation} with quadratic Hamiltonians as a lift of the flow $t\mapsto S^{H}_{t}$ in $\mathrm{Sp}(2n,\mathbb{R})$ into the unique path $t\mapsto \widehat{S}^{H}_{t}$ in $\mathrm{Mp}(2n, \mathbb{R})$ so that $\widehat{S}^{H}_{0}=I$.

\begin{corollary}
\label{corollary_quadractic_Hamiltonian_solution_quadratic_Fourier_transform}
Let the Hamiltonian $H$ be quadratic, and let $t\mapsto \widehat{S}^{H}_{t}$ be the lift to $\mathrm{Mp}(2n, \mathbb{R})$ of the flow $t\mapsto S^{H}_{t}$. Then for any $u_{0}\in \pazocal{S}(\mathbb{R}^{n})$, the function $u(x,t) = \widehat{S}^{H}_{t}u_{0}(x)$ is a solution of the initial value problem 
\begin{align*}
    i\hbar\frac{\partial u}{\partial t}(x,t) = \widehat{H}u(x,t), \ \ u(x,0) = u_{0}(x). 
\end{align*}
\end{corollary}

In particular, whenever $\exp(tJM)$ is a free symplectic matrix, we can express the solution of \eqref{Schrodinger_equation} as a quadratic Fourier transform \eqref{quadratic_Fourier_transform} of the initial condition.

\subsection{Hardy's Uncertainty Principle and the Wigner distribution}
\label{section_Hardys_UP_Wigner_distribution}
For a function $f:\mathbb{R}^{n}\to \mathbb{C}$, we normalize the Fourier transform according to \eqref{normalization_Fourier_transform}. Among the many uncertainty principles, Hardy's uncertainty principle is a precise statement regarding the largest possible decay of the pair $(f, \hat{f})$. The original 1933-paper \cite{Hardy_1933_A_Thm_Concerning_Fourier_Transforms} by Hardy covers the $1$-dimensional case, which has later been extended to higher dimensions (see \cite{FollandGerald1997TupA}). 

\begin{theorem}
\textnormal{(Hardy's Uncertainty Principle)} Suppose $f\in L^2(\mathbb{R}^{n})$ satisfies the decay conditions 
\begin{align*}
    |f(x)| \leq K e^{-\alpha |x|^2} \ \ \text{and} \ \ |\hat{f}(\xi)| \leq K e^{-\beta |\xi|^2}
\end{align*}
for some constants $\alpha,\beta,K>0$.
\begin{enumerate}[label =  \textnormal{(\roman*)},itemsep=0.4ex, before={\everymath{\displaystyle}}]
    \item If $4 \alpha \beta >1$, then $f\equiv 0$.
    \item If $4 \alpha \beta=1$, then $f=c e^{-\alpha |x|^2}$ for some $c\in\mathbb{C}$.
\end{enumerate}
\end{theorem}

In the joint time-frequency representation, there are analogous statements to Hardy's uncertainty principle. While these may originally have been formulated in terms of the \textit{Short-Time Fourier Transform} or the \textit{cross-Ambiguity function}, we shall present them equivalently in terms of the \textit{cross-Wigner distribution}. We normalize the cross-Wigner distribution using $\hbar$ to match our Weyl quantization procedure in \eqref{Weyl_quantization_procedure} for a \textit{physical system} (otherwise we can think of $\hbar=1$). Namely, for a pair $(f,g)\in \pazocal{S}\times \pazocal{S}'(\mathbb{R}^{n})$ the \textit{cross-Wigner distribution} is given by
\begin{align*}
    \pazocal{W}^{\hbar}(f,g)(x,\xi) := (2\pi\hbar)^{-n}\int_{\mathbb{R}^{n}} e^{-\frac{i}{\hbar} \langle \xi, y\rangle}f\left(x+\frac{y}{2}\right)\overline{g\left(x-\frac{y}{2}\right)}\mathrm{d}y. 
\end{align*}
In the case when $f = g$, we write $\pazocal{W}^{\hbar}(f,f) = \pazocal{W}^{\hbar}f$ and refer to it simply as the \textit{Wigner distribution} or \textit{Wigner transform of} $f$. The first analog of Hardy's uncertainty principle in the joint representation was discovered by Gröchenig and Zimmermann in \cite{Grochenig_Zimmermann_2001} (see Theorem 1.2 and Corollary 3.3). 
\begin{theorem}
\label{theorem_Hardys_UP_Wigner_joint_condition}
Suppose $(f,g)\in \pazocal{S}\times\pazocal{S}'(\mathbb{R}^{n})$ such that 
\begin{align*}
    |\pazocal{W}^{\hbar}(f,g)(x,\xi)| \leq K e^{-\left( \alpha|x|^2+\beta|\xi|^2\right)}
\end{align*}
for some constants $\alpha,\beta,K>0$.
\begin{enumerate}[label =  \textnormal{(\roman*)},itemsep=0.4ex, before={\everymath{\displaystyle}}]
    \item If $\alpha\beta>\hbar^{-2}$, then $\pazocal{W}^{\hbar}(f,g)\equiv 0$ so $f\equiv 0$ or $g\equiv 0$.
    \item If $\alpha\beta=\hbar^{-2}$ and $\pazocal{W}^{\hbar}(f,g)\not\equiv 0$, then both $f$ and $g$ are multiples of a time-frequency shift of the Gaussian $e^{-\frac{\alpha}{2}|x|^2}$, that is, $f$ and $g$ are multiples of $e^{i \langle \xi_0, x\rangle} e^{-\frac{\alpha}{2} |x-x_0|^2}$ for some constants $\xi_0, x_0 \in\mathbb{R}^n$.
\end{enumerate}
\end{theorem}
Later in \cite{Bonami_Demange_Jaming_Hardy_UP_2003}, several estimates on the largest possible decay of the Ambiguity function (or equivalently the Wigner distribution) have been derived. One of these results \textit{separates} the decay conditions in the $x$- and $\xi$-direction (see Corollary 6.5.). 

\begin{theorem}
\label{theorem_Hardy_UP_Wigner_distribution_separate_decay}
Suppose $f,g\in L^2(\mathbb{R}^n)$ such that 
\begin{align*}
    \int_{\mathbb{R}^{2n}}\frac{\left|\pazocal{W}^{\hbar}(f,g)(x,\xi) \ e^{2\pi|x_j|^2}\right|^2}{(1+|x_j|)^{M}}\mathrm{d}x\mathrm{d}\xi < \infty \ \ \text{and} \ \ \int_{\mathbb{R}^{2n}}\frac{\left|\pazocal{W}^{\hbar}(f,g)(x,\xi) \  e^{(2\pi)^{-1}\hbar^{-2}|\xi_j|^2}\right|^2}{(1+|\xi_j|)^{N}}\mathrm{d}x\mathrm{d}\xi < \infty
\end{align*}
for some $j=1,\dots,n$. If $\min\{M,N\} \leq 1$, then $\pazocal{W}^{\hbar}(f,g)\equiv 0$ so $f\equiv 0$ or $g\equiv 0$.
\end{theorem}
From the above result, we easily deduce sufficient decay conditions for the Wigner distribution similar to that of Theorem \ref{theorem_Hardys_UP_Wigner_joint_condition}, but with the $x$- and $\xi$-direction separated. 
 
\begin{corollary}
\label{corollary_Hardys_UP_Wigner_distribution}
Suppose $f,g \in L^2(\mathbb{R}^n)$ such that
\begin{align*}
    |\pazocal{W}^{\hbar}(f,g)(x,\xi)| \leq K e^{-\sum_{j}\alpha_j|x_j|^2} \ \ \text{and} \ \ |\pazocal{W}^{\hbar}(f,g)(x,\xi)| \leq K e^{-\sum_j \beta_j |\xi_j|^2}
\end{align*}
for some constants $\alpha_j,\beta_j,K>0$.
If for some $j=1,\dots,n$ the product $\alpha_j\beta_j>\hbar^{-2}$, then $\pazocal{W}^{\hbar}(f,g)\equiv 0$ so $f\equiv 0$ or $g\equiv 0$.
\end{corollary}

What is remarkable about the latest statement is that, although we require decay in every direction, it suffices to consider the \textit{largest} combined decay for \textit{any} given pair $(x_j, \xi_j)$ to conclude that the Wigner distribution is zero. Such conditions have also been derived for the separate representation in \cite{deGosson_Luef_2007}, \cite{deGosson_Luef_2009} based on the \textit{symplectic capacity} of the ellipsoid associated to the exponents.\footnote{In the separate representation, there is a much more general result stated for tempered distributions, see Corollary 1.6.9. in \cite{Demange2010UncertaintyPA}. Notably, if we restrict to functions in $L^2(\mathbb{R}^{n})$, it suffices to have large enough decay for one pair $(x_j,\xi_j)$ with \textit{no decay condition} in the other directions to conclude that the function is zero.} In the same vein, in \cite{deGosson_2021}, a similar result to Corollary \ref{corollary_Hardys_UP_Wigner_distribution} is obtained for the Wigner distribution. We shall utilize the uncertainty principle stated in Corollary \ref{corollary_Hardys_UP_Wigner_distribution} when we derive uniqueness results for the Schr\"{o}dinger equation with quadratic Hamiltonians. 

\section{Hardy's Uncertainty Principle for Quadratic Hamiltonians}
\label{section_Hardys_UP_for_quadratic_Hamiltonians}
We shall prove the following Hardy type estimate: 
\begin{theorem}
\label{theorem_Hardy_UP_quadratic_Hamiltonian}
Let $u(\cdot,t)\in \pazocal{S}(\mathbb{R}^{n})$ be the solution of the Schr\"{o}dinger equation \eqref{Schrodinger_equation}
with quadratic Hamiltonian $H(z) = -\frac{1}{2}\langle JX z, z\rangle$ for some $X\in \mathfrak{sp}(2n,\mathbb{R})$. Suppose at time $t=0$ and time $t=T$, the solution $u$ satisfies the decay conditions 
\begin{align*}
    |u(x,0)|\leq K e^{-\alpha |x|^2} \ \ \text{and} \ \ |u(x,T)|\leq K e^{-\beta |x|^2}
\end{align*}
for some constants $\alpha, \beta, K>0$. If the exponential
\begin{align*}
    \exp(TX) = \begin{pmatrix}\cdot & \mathcal{B}(T) \\ \cdot & \cdot \end{pmatrix} \ \ \text{is free symplectic and } \ \ (2\hbar)^2\|\mathcal{B}(T)\|_{\textnormal{op}}^{2}\ \alpha\beta >1,
\end{align*}
then $u\equiv 0$.
\end{theorem}
The proof is based on Corollary \ref{corollary_quadractic_Hamiltonian_solution_quadratic_Fourier_transform}, where the solution of the Schr\"{o}dinger equation can be written as a metaplectic transform of the initial condition. On this form, the proof is divided into two lemmas. The first lemma, Lemma \ref{lemma_Wigner_distribution_coordinate_transform_free_symplectic_group}, is well-known and referred to as the \textit{covariance property} of the Wigner distribution, where the Wigner distribution composed with a \textit{metaplectic} transformation corresponds to the Wigner distribution with an associated \textit{symplectic} coordinate transformation. This shows that the Wigner distribution of the solution equals, up the a symplectic coordinate transform, the Wigner distribution of the initial condition. In the second lemma, Lemma \ref{lemma_Hardys_UP_Wigner_transform_of_initial_condition}, we combine this fact with Hardy's uncertainty principle for the Wigner distribution.

\begin{lemma}
\label{lemma_Wigner_distribution_coordinate_transform_free_symplectic_group}
\textnormal{(Covariance property; see, e.g., Corollary 217 in \cite{deGosson_Symplectic_Methods})}
Let $\widehat{S}\in \mathrm{Mp}(2n,\mathbb{R})$, and let $S=\pi^{\mathrm{Mp}}(\widehat{S})$ denote the projection of $\widehat{S}$ on $\mathrm{Sp}(2n,\mathbb{R})$. Then for any $u\in \pazocal{S}(\mathbb{R}^n)$, the Wigner distribution of $\widehat{S}u$ satisfies 
\begin{align}
    \pazocal{W}^{\hbar}(\widehat{S}u) (x,\xi) = \pazocal{W}^{\hbar} u(S^{-1}(x,\xi)).
    \label{Wigner_distribution_coordinate_transform_equation}
\end{align}
\end{lemma}

Observe that by formula \eqref{symplectic_matrix_inverse}, the inverse of the free symplectic matrix $\exp(X) = \begin{psmallmatrix}\cdot & \mathcal{B}\\ \cdot & \cdot \end{psmallmatrix}$ is also free symplectic such that $\exp(X)^{-1} = \begin{psmallmatrix}\cdot  & -\mathcal{B}^{T}\\ \cdot & \cdot \end{psmallmatrix}$. Hence, Theorem \ref{theorem_Hardy_UP_quadratic_Hamiltonian} follows once we prove the next lemma. 

\begin{lemma}
\label{lemma_Hardys_UP_Wigner_transform_of_initial_condition}
Let $S^{-1}$ be a free symplectic matrix on the form 
\begin{align*}
    S^{-1} = \begin{pmatrix}A & B \\ C & D\end{pmatrix}, \ \ \det B\neq 0.
\end{align*}
Suppose that for two functions $u_0, u_1 \in \pazocal{S}(\mathbb{R}^{n})$ their Wigner distributions satisfy the identity
\begin{align*}
    \pazocal{W}^{\hbar}u_1(x,\xi) = \pazocal{W}^{\hbar}u_0\big(S^{-1}(x,\xi)\big). 
\end{align*}
Suppose further that the functions satisfy the decay conditions 
\begin{align*}
    |u_0(x)| \leq K e^{-\alpha |x|^{2}} \ \ \text{and} \ \ |u_{1}(x)| \leq K e^{-\beta |x|^{2}} \ \ \text{for} \ \ \alpha, \beta, K>0.
\end{align*}
If $(2\hbar)^2\alpha \beta\cdot \|B\|_{\textnormal{op}}^{2}>1$, then $u_0,u_1\equiv0$.
\begin{proof}
By the decay conditions on $u_0$ and $u_1$, we easily deduce that the associated Wigner distributions are bounded by
\begin{enumerate}[label =  \textnormal{(\roman*)},itemsep=0.4ex, before={\everymath{\displaystyle}}]
    \item $|\pazocal{W}^{\hbar}u_0(x,\xi)|\leq \kappa e^{-2\alpha |x|^2}$ and
    \item $|\pazocal{W}^{\hbar}u_1(x,\xi)|\leq \kappa e^{-2\beta |x|^2}$ for some constant $\kappa>0$.
\end{enumerate}
Similarly to the proof of Hardy's uncertainty principle for the free Schr\"{o}dinger equation in section 2.2 in \cite{Bertolin_Malinnikova_2021}, wherein the initial condition $u(x,0)$ is written as $u(x,0) = e^{ia|x|^{2}}f(x)$, we express
\begin{align*}
    u_{0}(x) = e^{\frac{i}{\hbar}\langle Mx, x\rangle} f(x) \ \ \text{for some} \ \ f\in \pazocal{S}(\mathbb{R}^n),
\end{align*}
and $M$ is some real $n\times n$ matrix to be decided.
Evidently, $|u_0(x)|=|f(x)|$, and we shall therefore prove the Hardy type estimate for $f$. On this form, the Wigner distribution of $u_0$ reads
\begin{align*}
    \pazocal{W}^{\hbar}u_0(x,\xi) = \pazocal{W}^{\hbar}f\left(x, \xi - (M+M^{T})x\right).
\end{align*}
Since the right-hand side of (i) is independent of $\xi$, we may set $\xi := \omega +(M+M^{T})x$ and maintain the same decay condition for $\pazocal{W}^{\hbar}f$ as for $\pazocal{W}^{\hbar}u_0$, namely 
\begin{align}
    |\pazocal{W}^{\hbar}f(x,\omega)| \leq \kappa e^{-2\alpha |x|^2}.
    \label{Wigner_distribution_f_decay_x_direction}
\end{align}
Similarly, we express the Wigner distribution of $u_1$ in terms of $f$, and utilizing the identity $\pazocal{W}^{\hbar}u_1(x,\xi) = \pazocal{W}^{\hbar}u_0(S^{-1}(x,\xi))$, we find that
\begin{align*}
    \pazocal{W}^{\hbar}u_1(x,\xi) = \pazocal{W}^{\hbar}f\big(Z(x,\xi)\big),
\end{align*}
where 
\begin{align*}
    Z = \begin{pmatrix}A & B\\ C-(M+M^{T})A & D - (M+M^{T})B\end{pmatrix}.
\end{align*}
We now choose the matrix $M$ such that $D - (M+M^{T})B = 0$, which is possible since $B$ is invertible. In particular, we have that
\begin{align*}
    M+M^{T} = DB^{-1},
\end{align*}
which shows that the matrix $DB^{-1}$ is symmetric, and in fact, we could have chosen $M$ to be symmetric. With this choice of $M$, the lower left block of $Z$ is given by $C-DB^{-1}A$. By the characterization of symplectic matrices \eqref{symplectic_block_form_conditions} and since $S^{-1}$ is \textit{free} symplectic, we may express $C$ in terms of $A,B$ and $D$, namely $C= (DB^{-1})^{T}A-(B^{-1})^{T}$. Thus, the lower left block simplifies to $-(B^{-1})^{T}$, and the matrix $Z$ in turn simplifies to
\begin{align*}
    Z = \begin{pmatrix}A & B\\ -(B^{-1})^{T} & 0\end{pmatrix},
\end{align*}
so that
\begin{align*}
    \pazocal{W}^{\hbar}u_1(x,\xi) = \pazocal{W}^{\hbar}f\left(Ax+B\xi, -(B^{-1})^{T}x\right).
\end{align*}
Again since the right-hand side of (ii) is independent of $\xi$, define $\xi := B^{-1}y-B^{-1}Ax$, so the decay condition reads
\begin{align*}
    |\pazocal{W}^{\hbar}f(y,-(B^{-1})^{T}x)|\leq \kappa e^{-2\beta |x|^2},
\end{align*}
or equivalently
\begin{align}
    |\pazocal{W}^{\hbar}f(y,\omega)|\leq \kappa e^{-2\beta |B^{T}\omega|^2}.
    \label{Wigner_distribution_f_decay_xi_direction}
\end{align}
By combining the two decay conditions \eqref{Wigner_distribution_f_decay_x_direction} and \eqref{Wigner_distribution_f_decay_xi_direction}, the statement now follows from Corollary \ref{corollary_Hardys_UP_Wigner_distribution}.
\end{proof}
\end{lemma}

\section{Examples of Schr\"{o}dinger evolutions}
\label{section_examples_of_Schrodinger_evolutions}
In this section we provide explicit examples of quadratic Hamiltonians and what the associated Hardy type estimate of Theorem \ref{theorem_Hardy_UP_quadratic_Hamiltonian} look like. 

\subsection{Free Schr\"{o}dinger equation and harmonic oscillator}
To begin with, we consider two cases where there are known Hardy type estimates. Namely, we consider the infamous free particle case and also a generalized harmonic oscillator, which has also been studied by Cassano and Fanelli in \cite{Cassano_Fanelli_2017} for the special case where all angular frequencies are equal. 

\begin{example}
(Free Schr\"{o}dinger equation) Consider a system without any external potential, that is, consider the Hamiltonian $H$ of the form
\begin{align*}
    H(z) = \frac{1}{2m}|p|^2 = \frac{1}{2m}\left(p_1^2+\dots+p_n^2\right), \ \ \text{where} \ \ z=(x,p). 
\end{align*}
This Hamiltonian corresponds to the so-called \textit{free} Schr\"{o}dinger equation
\begin{align}
    i\hbar \frac{\partial u}{\partial t}(x,t) = -\frac{\hbar^2}{2m}\Delta u(x,t). 
    \label{free_Schrodinger_equation}
\end{align}
Expressing the Hamiltonian instead as the inner product $H(z) = \frac{1}{2}\langle M z, z\rangle$, we have that
\begin{align*}
    M = \begin{pmatrix}0 & 0 \\ 0 & \frac{1}{m}I \end{pmatrix}, \ \ \text{and consequently} \ \ X:= J M = \begin{pmatrix}0 & \frac{1}{m}I\\ 0 & 0\end{pmatrix}\in \mathfrak{sp}(2n,\mathbb{R}). 
\end{align*}
Since $X^2 =0$, the exponential $\exp(tX)$ reduces to
\begin{align*}
    \exp(tX) = I + tX = \begin{pmatrix}I & \frac{t}{m}I\\ 0 & I\end{pmatrix},  
\end{align*}
which is free symplectic for all $t>0$. 
Hence, by Theorem \ref{theorem_Hardy_UP_quadratic_Hamiltonian}, we obtain the following statement for the free Schr\"{o}dinger equation: 
\begin{corollary}
Let $u(\cdot,t)\in \pazocal{S}(\mathbb{R}^{n})$ be the solution of the free Schr\"{o}dinger equation \eqref{free_Schrodinger_equation}.
Suppose at time $t=0$ and time $t=T$, the solution $u$ satisfies the decay conditions
\begin{align*}
    |u(x,0)|\leq K e^{-\alpha |x|^2} \ \ \text{and} \ \ |u(x,T)| \leq K e^{-\beta |x|^2} 
\end{align*}
for some constants $\alpha, \beta, K>0$. If $\alpha \beta \left(\frac{2\hbar T}{m}\right)^2>1$, then $u\equiv 0$.
\end{corollary}
\begin{remark}
By comparison with the already established Hardy's uncertainty principle for the free case, with $m=\frac{1}{2}$ and $\hbar =1$ (see Theorem 3 in \cite{Bertolin_Malinnikova_2021}), we have rediscovered the \textit{same} condition on the exponents $\alpha, \beta$. Furthermore, as the literature shows, our Hardy type estimate is, in fact, sharp. 
\end{remark}
\end{example}

\begin{example}
(Harmonic Oscillator) Consider now the Hamiltonian given by 
\begin{align*}
    H(z) = \frac{1}{2m}\left(p_1^2+\dots + p_n^2\right)+\frac{m}{2}\left(\omega_1^2 x_1^2+\dots + \omega_n^2 x_n^2\right), \ \ \text{where} \ \ z=(x,p). 
\end{align*}
This is known as the \textit{harmonic oscillator}, and the associated Schr\"{o}dinger equation reads
\begin{align}
    i\hbar \frac{\partial u}{\partial t}(x,t) = \left(-\frac{\hbar^2}{2m}\Delta + \frac{m}{2}\left(\omega_1^2 x_1^2+\dots + \omega_n^2 x_n^2\right)\right)u(x,t).
    \label{Schrodinger_equation_Harmonic_oscillator}
\end{align}
Define for simplicity the diagonal matrix $\Omega := \mathrm{diag}(\omega_j)$. Thus, on the inner product form $H(z)=\frac{1}{2}\langle Mz, z\rangle$, the matrix $M$ can be written
\begin{align*}
    &M = \begin{pmatrix} m\Omega^2 & 0 \\ 0 & \frac{1}{m}I \end{pmatrix}, \ \ \text{and also} \ \ X:=JM = \begin{pmatrix}0 & \frac{1}{m}I\\ -m\Omega^2 & 0\end{pmatrix} \in \mathfrak{sp}(2n,\mathbb{R}). 
\intertext{For the power series $\exp(tX) = \sum_{k=0}^{\infty} \frac{t^k}{k!}X^k$, we distinguish between the matrices with even and odd exponents, which, by induction, are given by}
    &X^{2k} = (-1)^{k}\begin{pmatrix} \Omega^{2k} & 0\\ 0 & \Omega^{2k}\end{pmatrix} \ \ \text{and} \ \ X^{2k+1} =(-1)^{k}\begin{pmatrix}0 & \frac{1}{m}\Omega^{2k}\\ -m \Omega^{2(k+1)}\end{pmatrix}.
\intertext{Thus, the summation reads} 
    &\exp(tX) = \sum_{k=0}^{\infty}\frac{(-1)^{k}t^{2k}}{(2k)!}\begin{pmatrix}\Omega^{2k} & 0\\ 0 & \Omega^{2k}\end{pmatrix} + \sum_{k=0}^{\infty}\frac{(-1)^{k}t^{2k+1}}{(2k+1)!}\begin{pmatrix}0 & \frac{1}{m}\Omega^{2k}\\ -m\Omega^{2(k+1)} & 0\end{pmatrix}.
\end{align*}
Since $\Omega$ is diagonal, we may move the summation inside the matrix, and we find that each block in $\exp(tX) = \begin{psmallmatrix}A(t) & B(t)\\ C(t) & D(t)\end{psmallmatrix}$ is rather easy to compute. In particular, for the $B(t)$-block, we have that
\begin{align*}
    B(t) = \frac{1}{m}\sum_{k=0}^{\infty}\frac{(-1)^{k}t^{2k+1}}{(2k+1)!}\Omega^{2k} = \frac{1}{m}\mathrm{diag}\left(\frac{1}{\omega_j}\sum_{k=0}^{\infty} \frac{(-1)^{k}(\omega_j t)^{2k+1}}{(2k+1)!}\right) = \frac{1}{m}\mathrm{diag}\left(\frac{\sin(\omega_j t)}{\omega_j}\right). 
\end{align*}
Similarly, for each of the other blocks, we recognize the Taylor series for the sine and cosine function, so that
\begin{align*}
    A(t) = D(t) = \mathrm{diag}\big(\cos(\omega_j t)\big) \ \ \text{and} \ \ C(t) = -m\ \mathrm{diag}\big(\omega_j\sin(\omega_j t)\big).
\end{align*}
Nonetheless, based solely on the $B(t)$-block and Theorem \ref{theorem_Hardy_UP_quadratic_Hamiltonian}, we obtain the following Hardy type estimate for the harmonic oscillator: 
\begin{corollary}
\label{corollary_Harmonic_oscillator_Hardy_UP}
Let $u(\cdot,t)\in\pazocal{S}(\mathbb{R}^{n})$ be the solution of the Schr\"{o}dinger equation \eqref{Schrodinger_equation_Harmonic_oscillator} corresponding to the harmonic oscillator. Suppose at time $t=0$ and time $t=T$, the solution $u$ satisfies the decay conditions
\begin{align*}
    |u(x,0)|\leq K e^{-\alpha |x|^2} \ \ \text{and} \ \ |u(x,T)| \leq K e^{-\beta|x|^{2}}
\end{align*}
for some constants $\alpha, \beta, K>0$. If 
\begin{align*}
    \sin(\omega_j T)\neq 0 \ \ \text{for all} \ \ j=1,2,\dots,n \ \ \text{and} \ \ \alpha\beta \left(\frac{2\hbar}{m}\right)^2\max{j} \left|\frac{\sin(\omega_j T)}{\omega_j}\right|^2>1, \ \ \text{then} \ \ u\equiv 0.
\end{align*}
\end{corollary}

\begin{remark}
Although we might presume that $\omega_j^2>0$, this is not actually a requirement. In Theorem \ref{theorem_Hardy_UP_quadratic_Hamiltonian}, we only require the matrix $M$ in $H(z) = \frac{1}{2}\langle Mz, z\rangle$ to be real-valued symmetric, which does not change if one or several of the $\omega_j^2<0$. The statement of Corollary \ref{corollary_Harmonic_oscillator_Hardy_UP} is not affected if $\omega_j^2<0$, other than, as a matter of preference, we choose to express the complex sine in terms of the real hyperbolic sine function, namely, $|\sin(i|\omega_j| T)| = |\sinh(|\omega_j|T)|$.
\end{remark}

\begin{remark}
In contrast to the free Schr\"{o}dinger equation, for the harmonic oscillator there are times $T>0$ where our procedure yields no Hardy type estimate. These are \textit{exactly} the time points when $\exp(TX)$ is no longer free symplectic, i.e., the time points when $\sin(\omega_j T)=0$. Since $\sinh(|\omega_j|T)\neq 0$ for all $T>0$, our Hardy type estimate only breaks down when $\omega_j>0$ for the discrete time points $T=\frac{k\pi}{\omega_j}$ for $k\in\mathbb{N}$. If $\omega_j=\omega$ for all $j$, the solution is \textit{periodic} in time, and these time points represent \textit{periods} or \textit{half periods} of the solution, for which we do not expect any Hardy type estimate to be present, as we are essentially trying to extract information from a a single time point. However, if $\omega_j-\omega_k \notin \mathbb{Q}$ for some pair $(\omega_j, \omega_k)$, the solution is not even periodic in time, and we do not have such a nice interpretation of why our Hardy estimate breaks down.   
\end{remark}

\begin{remark}
As previously mentioned, Hardy's uncertainty principle for the harmonic oscillator has also been studied by Cassano and Fanelli in \cite{Cassano_Fanelli_2017}, and with $m=\frac{1}{2}$, $\hbar = 1$ and $\omega_j^2 = \omega^2>0$ for $j=1,\dots,n$ they present a \textit{sharp} condition on the exponents $\alpha, \beta$ (see Theorem 1.3 and Theorem 1.9). By comparison with the above corollary, we again find that the condition coincides. 
\end{remark}
\end{example}

\subsection{Systems based on positive definite matrices}
From the previous two examples, it should be evident that providing an explicit Hardy type estimate based on Theorem \ref{theorem_Hardy_UP_quadratic_Hamiltonian} for a specific Hamiltonian $H(z) = \frac{1}{2}\langle Mz, z\rangle$ is contingent upon our ability to compute the associated exponential, $\exp(tJM)$. For the free and harmonic oscillator case, this is rather simple since $M$, in both cases, is a diagonal matrix. In general, this might be quite a challenging computation. However, for the family of positive definite matrices, we may apply Williamson's diagonalization theorem to simplify our computations. 

\begin{theorem}
\label{thm_Willamsons_diagonalization}
\textnormal{(Williamson's diagonalization theorem; see, e.g., Theorem 93 in \cite{deGosson_Symplectic_Methods})} Let $M$ be a positive definite symmetric real $2n\times 2n$-matrix. 
The eigenvalues of $JM$ are all of the form $\pm \lambda_j$ for $\lambda_j>0$, and the associated eigenvectors can be written $(e_j\pm i f_j)$ so that $\{e_j, f_j\}_{j=1}^{n}$ forms a symplectic basis. 
The matrix 
\begin{align}
    S := \left( e_1|\dots|e_n \ \Big| \ f_1|\dots|f_n\right)\in \mathrm{Sp}(2n,\mathbb{R})
\end{align} 
then diagonalizes $M$ such that
\begin{align}
    S^T M S 
    = \begin{pmatrix}
    \Lambda & 0\\
    0 & \Lambda
    \end{pmatrix} \ \ \text{for} \ \ \Lambda = \mathrm{diag}(\lambda_j).
\end{align}
\end{theorem}
From here, we obtain a closed form of the exponential.
\begin{lemma}
\label{lemma_Exponential_symplectic_Lie_algebra_positive_definite_matrix}
Let $M$ be a positive definite symmetric real $2n\times 2n$ matrix, and define the matrix $X:=JM$. The eigenvalues of $X$ are on the form $\pm i\lambda_j$ for $\lambda_j>0$, and the associated eigenvectors can be written $(e_j\pm i f_j)$ so that $\{e_j, f_j\}_{j=1}^{n}$ forms a       symplectic basis. Then for the matrix 
\begin{align}
    S:=\left(e_1 | \dots| e_n \ \Big| \ f_1|\dots|f_n\right)\in \mathrm{Sp}(2n,\mathbb{R}),
    \label{Exponential_sympletic_Lie_algebra_diagonalization_matrix_construction}
\end{align}
we have the following matrix decomposition of the exponential
\begin{align}
    \exp(tX) = J (S^{-1})^{T} \begin{pmatrix} \Theta&  -\Omega \\ \Omega&  \Theta
    \end{pmatrix} S^{-1}
    \label{Exponential_symplectic_Lie_albegra_positive_definite_matrix_equation}
\end{align}
with $\Theta := \mathrm{diag}(\sin(\lambda_jt))$, $\Omega := \mathrm{diag}(\cos(\lambda_j t))$ and parameter $t\in \mathbb{R}$.
\begin{proof}
By the diagonalization in Theorem \ref{thm_Willamsons_diagonalization}, we can write 
\begin{align*}
    &X = J (S^{-1})^{T}\begin{pmatrix} \Lambda & 0\\ 0 & \Lambda \end{pmatrix}S^{-1} \ \ \text{for} \ \ \Lambda = \mathrm{diag}(\lambda_j). 
\intertext{By definition of symplectic matrices, $S^{-1} J (S^{-1})^{T} = J$, it follows that}
    &X^{j} = J (S^{-1})^{T}\begin{pmatrix} \Lambda & 0 \\ 0 & \Lambda \end{pmatrix} \left[J \begin{pmatrix} \Lambda & 0 \\ 0 & \Lambda \end{pmatrix}\right]^{j-1}S^{-1} \ \ \text{for} \ \ j =1,2,\dots
\end{align*}
In turn, the power series of $\exp(tX)$ reads 
\begin{align*}
    \exp(tX) &= \sum_{j=0}^{\infty}\frac{t^j X^j}{j!} = I + J(S^{-1})^{T}\left[\sum_{j=1}^{\infty} \frac{t^j}{j!} \begin{psmallmatrix} \Lambda & 0 \\ 0 & \Lambda \end{psmallmatrix} \left[J\begin{psmallmatrix} \Lambda & 0 \\ 0 & \Lambda \end{psmallmatrix}\right]^{j-1}\right]S^{-1}.
\end{align*}
We now compute the sum within the square brackets $[\dots]$, and similarly to the harmonic oscillator, we distinguish between even and odd exponents, where
\begin{align*}
    \left[J \begin{psmallmatrix} \Lambda & 0 \\ 0 & \Lambda \end{psmallmatrix}\right]^{2k} = (-1)^{n}\begin{psmallmatrix} \Lambda^{2k} & 0 \\ 0 &  \Lambda^{2k}\end{psmallmatrix} \ \ \text{and} \ \ \left[J \begin{psmallmatrix} \Lambda & 0 \\ 0 & \Lambda \end{psmallmatrix}\right]^{2k+1} = (-1)^{n}\begin{psmallmatrix} 0 & \Lambda^{2k+1} \\ -\Lambda^{2k+1} & \end{psmallmatrix}.
\end{align*}
Thus, the sum can be expressed
\begin{align*}
   &\sum_{j=1}^{\infty} \frac{t^j}{j!} \begin{psmallmatrix} \Lambda & 0 \\ 0 & \Lambda \end{psmallmatrix} \left[J\begin{psmallmatrix} \Lambda & 0 \\ 0 & \Lambda \end{psmallmatrix}\right]^{j-1}=\sum_{k=0}^{\infty}\frac{(-1)^{k} t^{2k+1}}{(2k+1)!}\begin{psmallmatrix} \Lambda^{2k+1} & 0 \\ 0 & \Lambda^{2k+1} \end{psmallmatrix} - \sum_{k=1}^{\infty} \frac{(-1)^{k}t^{2k}}{(2k)!}\begin{psmallmatrix} 0 & \Lambda^{2k} \\ -\Lambda^{2k} & 0 \end{psmallmatrix}.
\end{align*}
By pulling the summation inside the matrices, we immediately recognize the Taylor series of the sine and cosine function. In particular, with the same notation as in \eqref{Exponential_symplectic_Lie_albegra_positive_definite_matrix_equation}, we find that
\begin{align*}
   &\sum_{j=1}^{\infty} \frac{t^j}{j!} \begin{psmallmatrix} \Lambda & 0 \\ 0 & \Lambda \end{psmallmatrix} \left[J\begin{psmallmatrix} \Lambda & 0 \\ 0 & \Lambda \end{psmallmatrix}\right]^{j-1}=\begin{pmatrix} \Theta & 0 \\ 0 & \Theta \end{pmatrix} +J - \begin{pmatrix} 0 & \Omega \\ -\Omega & 0 \end{pmatrix}.
\end{align*}
Again since $S^{-1}$ is symplectic, we have that $J\big((S^{-1})^{T}J S^{-1}\big) = J^2 = -I$, and result \eqref{Exponential_symplectic_Lie_albegra_positive_definite_matrix_equation} follows. 
\end{proof}
\end{lemma}

The purpose of the third and final example is twofold: Firstly, it illustrates how we may apply Lemma \ref{lemma_Exponential_symplectic_Lie_algebra_positive_definite_matrix} to produce explicit Hardy type estimates based on Theorem \ref{theorem_Hardy_UP_quadratic_Hamiltonian}. Secondly, the example does not seem to be covered by previous literature. 

\begin{example}
(2D--harmonic oscillator with a cross-term) Consider the Hamiltonian 
\begin{align*}
    H(z) = \frac{1}{2m}\left(p_1^2+p_2^2\right) + \theta p_1 x_2 + \frac{m\omega^2}{2}\left(x_1^2+x_2^2\right) \ \ \text{where} \ \ z=(x,p)\in\mathbb{R}^{2}\times\mathbb{R}^{2},
\end{align*}
which corresponds to the Schr\"{o}dinger equation
\begin{align}
    i\hbar \frac{\partial u}{\partial t}(x,t) = \left(-\frac{\hbar^2}{2m}\Delta -i\hbar \ \theta x_2\frac{\partial}{\partial x_1}+ \frac{m\omega^2}{2}\left(x_1^2 + x_2^2\right)\right)u(x,t).
    \label{Schrodinger_equation_with_cross_term}
\end{align}
This system is similar to the harmonic oscillator but with an added cross-term $+\theta p_1 x_2$ in the Hamiltonian. With this additional term, the symmetric matrix $M$ in $H(z) = \frac{1}{2}\langle Mz, z\rangle$ is no longer a diagonal matrix, rather, we have that
\begin{align*}
    M = &\begin{pmatrix}m\omega^2 & 0 & 0 & 0\\ 0 & m\omega^2 & \theta & 0\\ 0 & \theta & \frac{1}{m} & 0\\ 0 & 0 & 0 & \frac{1}{m}\end{pmatrix},
    \intertext{in addition to}
    X:=JM = &\begin{pmatrix}0 & \theta & \frac{1}{m} & 0 \\ 0 & 0 & 0 & \frac{1}{m}\\ -m\omega^2 & 0 & 0 & 0\\ 0 & -m\omega^2 & -\theta & 0\end{pmatrix}\in \mathfrak{sp}(4,\mathbb{R}).
\end{align*}
Proceeding, we assume $\theta\in\mathbb{R}$ such that $\omega>|\theta|$, making $M$ a real positive definite matrix, to which we may apply Lemma \ref{lemma_Exponential_symplectic_Lie_algebra_positive_definite_matrix}. Consider first the following set of vectors in $\mathbb{R}^{4}$:
\begin{equation}
    \begin{aligned}
        &e_1 := \left(\frac{m}{2}\sqrt{\omega(\omega+\theta)}\right)^{\frac{1}{2}}\begin{pmatrix}-\frac{1}{m\omega}\\ 0 \\ 0 \\ 1\end{pmatrix}, \ \ 
    f_1 := \left(\frac{m}{2}\sqrt{\omega(\omega+\theta)}^{-1}\right)^{\frac{1}{2}}\begin{pmatrix}0\\ -\frac{1}{m} \\ -\omega \\ 0\end{pmatrix},\\
    &e_2 := \left(\frac{m}{2}\sqrt{\omega(\omega-\theta)}\right)^{\frac{1}{2}}\begin{pmatrix}-\frac{1}{m\omega}\\ 0 \\ 0\\ -1\end{pmatrix}, \ \ 
    f_2 := \left(\frac{m}{2}\sqrt{\omega(\omega-\theta)}^{-1}\right)^{\frac{1}{2}}\begin{pmatrix}0\\ \frac{1}{m}\\ -\omega \\ 0 \end{pmatrix}.
    \end{aligned}
    \label{eigenvectors_real_imaginary_part_example}
\end{equation}
It is straightforward to verify that $(e_j\pm i f_j)$ constitute eigenvectors of $X$ such that
\begin{align}
    X (e_j\pm i f_j) = \pm i \lambda_j (e_j\pm i f_j), \ \ \text{where} \ \ \lambda_1 = \sqrt{\omega(\omega+\theta)} \ \ \text{and} \ \ \lambda_2 = \sqrt{\omega(\omega-\theta)}.
    \label{eigenvalues_example}
\end{align}
By the normalization of the vectors, we also find that $\{e_j, f_j\}_{j=1}^{2}$ forms a symplectic basis for $\mathbb{R}^2\times \mathbb{R}^2$. Thus, we may use
\begin{align}
    S:= \big(e_1|e_1 | f_1 | f_2\big)\in \mathrm{Sp}(4,\mathbb{R})
    \label{decomposition_matrix_example}
\end{align}
to decompose the exponential, $\exp(tX)$, according to Lemma \ref{lemma_Exponential_symplectic_Lie_algebra_positive_definite_matrix} such that
\begin{align*}
    \exp(tX) = J (S^{-1})^{T}\begin{pmatrix}\Theta & -\Omega\\ \Omega & \Theta \end{pmatrix}S^{-1} \ \ \text{with} \ \ \Theta :=\mathrm{diag}(\sin(\lambda_j t)) \ \ \text{and} \ \ \Omega:=\mathrm{diag}(\cos(\lambda_j t)). 
\end{align*}
Writing $S$ on block form $S=\begin{psmallmatrix}A & B\\ C & D\end{psmallmatrix}$ and utilizing expression \eqref{symplectic_matrix_inverse} for the inverse $S^{-1}$, it follows that
\begin{align*}
     \mathcal{B}(t) = B\Theta B^{T} + B \Omega A^{T} -A\Omega B^{T} + A\Theta A^{T}, \ \ \text{where} \ \ \exp(t X) = \begin{pmatrix}\cdot & \mathcal{B}(t)\\ \cdot & \cdot \end{pmatrix}.
\end{align*}
By combining \eqref{eigenvectors_real_imaginary_part_example}--\eqref{decomposition_matrix_example}, the two upper blocks read
\begin{align*}
    A = -\begin{pmatrix}\sqrt{\frac{\lambda_1}{2m\omega^2}} & \sqrt{\frac{\lambda_2}{2m\omega^2}}\\ 0 & 0\end{pmatrix} \ \ \text{and} \ \ B= \begin{pmatrix}0 & 0\\ -\frac{1}{\sqrt{2m\lambda_1}} & \frac{1}{\sqrt{2m\lambda_2}}\end{pmatrix}.
\end{align*}
Subsequently, we obtain an \textit{explicit} expression for the upper right block of the exponential, $\exp(tX)$, namely
\begin{align}
    \mathcal{B}(t) = \frac{1}{2m\omega}\begin{pmatrix}\frac{\lambda_1}{\omega} \sin(\lambda_1 t) + \frac{\lambda_2}{\omega}\sin(\lambda_2 t) & \cos(\lambda_2 t)-\cos(\lambda_1 t) \\ \cos(\lambda_1 t)-\cos(\lambda_2 t) & \frac{\omega}{\lambda_1}\sin(\lambda_1 t)+\frac{\omega}{\lambda_2}\sin(\lambda_2 t)\end{pmatrix},
    \label{upper_right_block_matrix_explicit_example}
\end{align}
With this latest result, we are ready to compute for the Schr\"{o}dinger equation \eqref{Schrodinger_equation_with_cross_term} an \textit{explicit} Hardy type estimate based on Theorem \ref{theorem_Hardy_UP_quadratic_Hamiltonian}.
\begin{corollary}
Let $u(\cdot,t)\in\pazocal{S}(\mathbb{R}^{n})$ be the solution of the Schr\"{o}dinger equation \eqref{Schrodinger_equation_with_cross_term}, with $\theta\in\mathbb{R}$ such that $\omega>|\theta|$. Suppose at time $t=0$ and time $t=T$, the solution $u$ satisfies the decay conditions
\begin{align*}
    |u(x,0)|\leq K e^{-\alpha |x|^2} \ \ \text{and} \ \ |u(x,T)| \leq K e^{-\beta|x|^{2}}
\end{align*}
for some constants $\alpha, \beta, K>0$. Define the quantity 
\begin{equation}
\begin{aligned}
    &2(2m\omega)^2\cdot \Gamma_{\pm}(T):=\frac{\omega^4+\lambda_{1}^4}{\omega^2\lambda_{1}^2}\sin^2(\lambda_1 T)+\frac{\omega^4+\lambda_{2}^4}{\omega^2\lambda_{2}^2}\sin^2(\lambda_2 T)\\
    &+2\frac{\omega^4+(\lambda_1\lambda_2)^2}{(\omega\lambda_1\lambda_2)^2}\sin(\lambda_1 T)\sin(\lambda_2 T)+2\big(\cos(\lambda_1 T)-\cos(\lambda_2 T)\big)^2\\
    &\pm\left|\frac{\omega^2-\lambda_{1}^2}{\omega\lambda_{1}}\sin(\lambda_1 T)+\frac{\omega^2-\lambda_{2}^2}{\omega\lambda_{2}}\sin(\lambda_2 T)\right|\ \cdot\\
    &\cdot\left(\left(\frac{\omega^2+\lambda_1^2}{\omega \lambda_{1}}\sin(\lambda_{1}T)+\frac{\omega^2+\lambda_{2}^2}{\omega \lambda_{2}}\sin(\lambda_{2}T)\right)^2+4\big(\cos(\lambda_1 T)-\cos(\lambda_2 T)\big)^2\right)^{\frac{1}{2}},
    \label{corollary_quantity_definition_final_example}
\end{aligned}
\end{equation}
with $\lambda_1 = \sqrt{\omega(\omega+\theta)}$ and $\lambda_2 = \sqrt{\omega(\omega-\theta)}$. If $\Gamma_{\pm}(T)\neq 0$ and $\alpha\beta (2\hbar)^2\cdot \Gamma_{+}(T) >1$, then $u\equiv 0$.
\end{corollary}
\begin{proof}
It remains to show that the eigenvalues of $\mathcal{B}^{T}\mathcal{B}(T)$ coincide with $\Gamma_{-}(T)\leq \Gamma_{+}(T)$ so that $\Gamma_{+}(T) = \| \mathcal{B}(T)\|_{\text{op}}^{2}$. From \eqref{upper_right_block_matrix_explicit_example}, we see that the matrix $\mathcal{B}(T)$ is of the form $\mathcal{B}(T) = \begin{psmallmatrix}a & b\\ -b & c\end{psmallmatrix}$, and in turn
\begin{align*}
    \mathcal{B}^{T}\mathcal{B}(T) = \begin{pmatrix}a^2+b^2 & (a-c)b\\ (a-c)b & b^2+c^2\end{pmatrix}.
\end{align*}
The two eigenvalues $\delta_{\pm}$ of $\mathcal{B}^{T}\mathcal{B}(T)$ are then given by
\begin{align*}
    \delta_{\pm} = \frac{1}{2}\Big[\left(a^2+2b^2+c^2\right)\pm|a-c|\sqrt{(a+c)^2+4b^2}\Big],
\end{align*}
and the result follows once we replace $a,b$ and $c$ with the expressions in \eqref{upper_right_block_matrix_explicit_example}.
\end{proof}

\begin{remark}
Although the quantity $\Gamma_{\pm}(T)$ in \eqref{corollary_quantity_definition_final_example} looks rather daunting, it is easy to see that the expression reduces to the harmonic oscillator case of Corollary \ref{corollary_Harmonic_oscillator_Hardy_UP} when $\theta\to0$. 
\end{remark}

\end{example}

\section*{Acknowledgements}
The research of the author was supported by Grant 275113 of the Research Council of
Norway. The author would like to extend thanks to prof. Eugenia Malinnikova and prof. Franz Luef for the many insightful discussions and feedback on early drafts of manuscript.

\addcontentsline{toc}{section}{References}
\bibliographystyle{abbrv}
\bibliography{Notes_on_Hardys_UP_for_Wigner_distribution_and_Schrodinger_evolutions_HelgeKnutsen}

\end{document}